%% file: ms.tex
\pdfoutput=1 

\documentclass[11pt]{amsart}

\usepackage{style}
\include{config}

\include{macros}

\begin{document}

\begin{abstract}
	The functor that takes a ring to its category of modules
	has an adjoint if one remembers the forgetful functor to
	abelian groups: the \emph{endomorphism ring} of linear
	natural transformations.
	This uses the self-enrichment of the category of abelian groups.
	If one considers enrichments
	into symmetric sequences or even bisymmetric sequences, one
	can produce an \emph{endomorphism operad} or an \emph{endomorphism
	properad}.

	In this note, we show that more generally, given an category
	\( \ground \) enriched in a monoidal category \( \moncat \), the
	functor that associates to a monoid in \( \moncat \) its category
	of representations in \( \ground \) is adjoint to the functor
	that computes the \emph{endomorphism monoid} of any functor
	with domain \( \ground \).
	After describing the first results
	of the theory we give several examples of applications.
\end{abstract}

\maketitle

The functor that takes a ring \( R \) to its category of modules has an
adjoint, provided that in addition to \( \catofmod R \), one
remembers the forgetful functor
\[
	\catofmod R \longrightarrow \abelian.
\]
The adjoint sends a functor \( F \from \domain \to \abelian \) to its
\emph{endomorphism ring} \( \auto F \) of natural transformations. 
This fact is familiar to people working on duality results à la Tannaka.

If instead of using the self-enrichment \( \langle -, - \rangle \from
\abelian\op \times \abelian \to \abelian \), one uses an enrichment into
symmetric sequences or bisymmetric sequences, then \( \auto F \) can be
promoted to an \emph{endomorphism operad} or an \emph{endomorphism
properad}.
This is summarized in the table:

\begin{table}[h]
	\begin{tabular}{@{}SlSc@{}} 
		\toprule
		\multicolumn{1}{Sc}{\( \auto F \)} & enrichment
		\\
		\midrule
		endomorphism ring  & \( \langle X,Y\rangle \)
		\\
		endomorphism operad & \( \langle X^{\otimes n},Y\rangle \) 
		\\
		endomorphism properad
		& \( \langle X^{\otimes p},Y^{\otimes q}\rangle \)
		\\
		\bottomrule
	\end{tabular}
\end{table}

In this note we study the general case, replacing \( \abelian \) by a
category \( \ground \) enriched in a monoidal category \( \moncat \).
First we review representations of monoids in the 
context of an enriched category.
Then we describe the \emph{endomorphism monoid} of a 
functor whose target is an enriched category and show that this
construction is adjoint to the representations functor.

After describing the adjunction between monoids in \( \moncat \) and
functors with target \( \ground \), we shall study the basic
properties of this adjunction, in particular in the case where the
enrichment is also tensored. 

In two brief appendices, we provide quick definitions of terms in
enriched category theory that we need and give a few examples of
contexts in which this setup holds. 

The sequel, \emph{Endomorphism operads of functors}%
~\cite{arXiv:190609006D}, 
contains some explicit computations.

After seeing the definitions of the functors \( \autofunctor \)
and \( \repfunctor \) and their adjunction,
the reader is encouraged to take a look at the appendix%
~[\sectionref{Appendix: contexts}].
Some of the examples there might be surprising.

\section{Monoids and their representations}%
\label{section:monoids and their representations}

Let us fix a a bicomplete monoidal category \( \moncat \) and a category
\( \ground\) enriched in \(\moncat \):
\[
	\begin{tikzcd}
		\ground\op \times \ground \ar[rr,"{[-,-]}"] && \moncat.
	\end{tikzcd}
\]
For convenience, we shall assume given a \emph{locally large} universe
enlargement \( \moncat \hookrightarrow \biguni \)%
~[\sectionref{Appendix: enlargement of the universe}].
Because \( \moncat \hookrightarrow \biguni \) is fully faithful
and monoidal, one
has a fully faithful embedding of categories of monoids
\[
	\begin{tikzcd}
		\monoids{\moncat} \rar[hook] & \monoids{\biguni}.
	\end{tikzcd}
\]
In order to distinguish between the two, we shall say that a monoid in
\( \biguni \) is \emph{large}.

\begin{remark}[Endomorphism monoid of an object]
	Thanks to the \( \moncat \)\=/enrichment of \( \ground \),
	every object \(X \in \ground\)
	has a natural endomorphism monoid \( [X, X] \).
\end{remark}

\begin{definition}[Representations of monoids]%
	\label{definition:reps}
	Let \( M \) be a monoid.
	Its category of representations in \( \ground \)
	\[
		\rep M
	\]
	is the large category
	\begin{itemize} 
		\item whose objects are \( (X,\alpha) \)
		      where \( X \) is an object of
		      \( \ground \)
		      and \( \alpha \from M\to [X,X] \) is a map of monoids and
		\item whose morphisms \( (X,\alpha) \to (Y,\beta) \)
		      are maps \( f \from X \to Y \)
		      such that the following diagram commutes:
		      \[
		      	\begin{tikzcd}[ampersand replacement=\&]
		      		M
		      		\arrow[r, "\alpha"]
		      		\arrow[d, "\beta", swap]
		      		\& {[X,X]}
		      		\arrow[d, "f_\ast"] \\
		      		{[Y,Y]}
		      		\arrow[r, "f^\ast",swap]
		      		\& {[X, Y]}.
		      	\end{tikzcd}
		      \]
	\end{itemize}
\end{definition}

The category of representations of \( M \)
has an evident forgetful functor
\[
	\forget M \from \rep M \longrightarrow \ground
\]
that is both faithful and conservative. The assignment \(M \mapsto
\rep M\) is moreover functorial: given a morphism of monoids
\( \psi \from M \to N \), one has a commutative diagram
\[
	\begin{tikzcd}
		\rep M \ar[rd, "\forget M"'] && \rep N \ar[ld, "\forget N"]
		\ar[ll, "\forget \psi"']
		\\
		& \ground &
	\end{tikzcd}
\]

Denoting by \( \largecats \) the very large category
of large categories, one gets a \emph{representation functor}
\[
	\begin{tikzcd}
		\monoids{\moncat} \ar[rr, "\repfunctor"] &&
		\left(\slice{\largecats}{\ground}\right)\op.
	\end{tikzcd}
\]

\begin{remark}[Representations of large monoids]
	Since we have required \( \biguni \) to be \emph{locally large},
	the definition of the category of representations \( \rep M \)
	also makes sense for \( M \) a large monoid.
	Then, the large category \( \ground \) having been fixed,
	the representations functor extends to the category
	of large monoids:
	\[
		\begin{tikzcd}
			\monoids{\moncat\vphantom{\biguni}} \ar[rr, "\repfunctor"]
			\dar[hook] &&
			\left(\slice{\largecats}{\ground}\right)\op.
			\\
			\monoids{\biguni} \ar[rru, dashed] &&
		\end{tikzcd}
	\]
	Indeed, let \( M \) be a large monoid. 
	The cardinality of the objects of \( \rep M \) is bounded by
	\[
		\bigcup_{X \in \ground} \hom {\biguni} M {[X,X]}.
	\]
	Since \( \ground\) is large and \( \biguni \) is locally large,
	we deduce that \( \rep M \) has a large set of objects.
	Given two representations
	\( X \) and \( Y \) of a monoid \( M \), one has
	\[
		\hom {\rep M} X Y
		\subset \hom \ground {\forget M X} {\forget M Y}.
	\]
	Hence, since \( \ground\) has large sets of morphisms,
	so does \(\rep M \).
\end{remark}

\section{The endomorphism monoid of a functor}%
\label{section: The endomorphism monoid of a functor}

In this section we show that the representation functor
\(M \mapsto \rep M\) has a right adjoint
\[
	\left(
	\begin{tikzcd}
		\domain
		\ar[d, "F"]
		\\
		\ground
	\end{tikzcd}
	\right)
	\longmapsto \auto F.
\]
It takes as inputs large categories \( \domain \) over \( \ground \)
and outputs the \emph{endomorphism monoid}
\( \auto F \) of the functor \( F \from \domain \to \ground \).

\begin{remark}[Enriched natural transformations]
	Given a large category \( \domain \),
	the category of functors \(\fun ({\domain}, \ground)\)
	is naturally enriched in \( \biguni \) as follows.
	Given two functors \( F, G \from \domain \to \ground\),
	the \(\moncat \)-natural transformations from \( F \) to \( G\)
	are presented by the object of \(\biguni \) given by
	\[
		\nat \moncat F G \coloneqq \coint {\domain}
		[F -, G -],
	\]
	where, following Yoneda's original notation%
	~\cite[\S 4]{On_Ext_and_exact_sequences}, 
	\( \coint {\domain} \) denotes the \emph{cointegration} (or end) 
	of a functor \( \domain\op \times \domain \to \ground \).
\end{remark}

\begin{definition}%
	\label{definition: object of endos, closed case}
	The \emph{endomorphism monoid} of a functor
	\(F \from \domain \to \ground\) is
	\[
		\auto F \coloneqq \nat \moncat F F
	\]
	the (large) monoid
	of \( \moncat \)-natural transformations of \( F \).
\end{definition}

\begin{remark}[Functoriality of \( \autofunctor \)]
	As is the case in any 2-categorical setting,
	\( \moncat \)-natural transformations are
	compatible with `horizontal composition’ or `whiskering':
	\[
		\left(\begin{tikzcd}
			\domain' \rar["\Phi"]
			& \domain
			\ar[rr, shift left=3.5, phantom, ""{name=U, below}]
			\ar[rr, shift left=3.5, "F"]
			\ar[rr, shift right=3.5, "F"']
			\ar[rr, shift right=3.5, phantom, ""{name=D, above}]
			&& \ground
			\ar[Rightarrow, from=U, to=D]
		\end{tikzcd}\right)
		\longmapsto
		\left(\begin{tikzcd}
			\domain'  \ar[rr, shift left=3.5, ""{name=U, below}]
			\ar[rr, shift left=3.5, "F \circ \Phi"]
			\ar[rr, shift right=3.5, "F \circ \Phi"']
			\ar[rr, shift right=3.5, ""{name=D}]&& \ground
			\ar[Rightarrow, from=U, to=D]
		\end{tikzcd}\right).
	\]
	Thus, the construction \( F \mapsto \auto F \) is functorial
	in the sense that given
	\[
		\begin{tikzcd}
			\domain'\ar{r}{\Phi}
			&\domain \rar["F"] & \ground,
		\end{tikzcd}
	\]
	one gets a morphism of large monoids
	\[
		\Phi^\ast \from \auto F \longrightarrow \auto {F \circ \Phi}.
	\]
\end{remark}

\begin{theorem}%
	\label{theorem: adjunction}
	The functor \( \autofunctor \)
	is right adjoint to \( \repfunctor \)
	\[
		\begin{tikzcd}[ampersand replacement=\&]
			\monoids{\biguni} \arrow[rr, shift left=2,"\repfunctor"]
			\&
			\&
			\left(\slice{\largecats} \ground\right)\op.
			\arrow[ll, shift left=2, "\autofunctor"]
		\end{tikzcd}
	\]
\end{theorem}
There are a number of examples where this setup gives interesting
endomorphism monoids and interesting adjunctions~[\sectionref{Appendix:
contexts}].

\begin{proof}
	Observe that a functor from \( \domain \) to \( \rep M \)
	over \(\ground \) consists of:
	\begin{itemize} 
		\item at the object level, a monoid map \(M\xrightarrow{\psi_X}
		      [F(X),F(X)]\) for each object \( X \) of \( \domain \),
		      and
		\item at the morphism level, no data,
		      since the value on morphisms
		      is determined by being over \( \ground \)
		      and the functor from \( \rep M \) to \( \ground \)
		      is faithful.
	\end{itemize}
	However, to be a functor,
	the collection \( \psi_X \) must satisfy a condition
	so that for each map \( f \) in \( \hom {\domain} X Y \),
	the map \( F(f) \) is an \( M \)-representation map
	between \( F(X) \) and \( F(Y) \).
	This is precisely the condition for the maps \( \psi_X \)
	to assemble to a map \( \psi \from M \to \auto F \).
	Compatibility with the \( M \)\=/representation structures
	for each object \( X \)
	implies that \( \psi \) is a morphism of large monoids.
\end{proof}

\begin{example}%
	\label{example: object}
	Let \( X \from * \to \ground \) be an object of \( \ground \).
	Then the equalizer formula
	for the cointegral computing \( \auto X \) collapses to \( [X,X] \).
	So in this case \( \auto X \)
	recovers the ordinary endomorphism object \( [X,X] \).
\end{example}

\begin{example}%
	\label{example: morphism}
	Let \( f\from \upDelta^{\!1}\to\ground \)
	be a morphism of \( \ground \),
	with domain \( X \) and codomain \( Y \).
	Again the cointegral has a simple description
	via the equalizer formula;
	it is the pullback of \( [X,X]\) and \([Y,Y] \) over
	\( [X,Y] \).
	\[
		\auto f\isonat [X,X]
		\bigtimes_{[X,Y]}[Y,Y].
	\]
	This is sometimes called the endomorphism monoid of \( f \)%
	~\cite[13.10]{isbn:978-1-4704-2723-8}.
\end{example}

\begin{remark}[Generalized enrichments]
	We have taken as our fundamental input an enrichment
	of the category \( \ground \)
	in the monoidal category \( \moncat \).
	A generalization of this framework
	is to consider instead a lax functor
	\[
		\ground \longrightarrow \pointedbim\moncat
	\]
	where \( \pointedbim\moncat \) is the bicategory
	whose objects are monoids in \( \moncat \),
	whose morphisms are pointed bimodules,
	and whose \( 2 \)-morphisms are maps of bimodules.
	
	Let us present an example of such a generalized enrichment that
	does not fit directly in our framework.
	Let \( \ground \) be a large category, seen as naturally enriched in
	large sets.
	There is a lax functor
	\[
		\ground \longrightarrow \pointedbim\bigsets
	\]
	given on objects by
	\[
		X \longmapsto \Aut(X),
	\]
	which sends a map \( f \from X \to Y \) to
	\[
		{\Hom(X,Y)}_f \coloneqq \Hom(X,Y)\text{ pointed by }f
	\]
	and which sends the composite of two maps \( f \) and \( g \) to
	\[
		{\Hom(X,Y)}_f \otimes_{\Aut(Y)} {\Hom(Y,Z)}_g \longrightarrow
		{\Hom(X,Z)}_{fg}.
	\]
	Using the same ideas,
	one can see how to produce a generalized enrichment
	out of a \( \moncat \)-enriched category \( \ground \) via
	\[
		X \longmapsto [X, X].
	\]
	
	The cointegral defining the endomorphism monoid
	of a functor \( F \) has a
	natural extension to the generalized framework.
	
	The generalized enrichment of our example
	yields the following adjunction
	\[
		\begin{tikzcd}[ampersand replacement=\&]
			\biggroups \arrow[rr, shift left=2,"\repfunctor"]
			\&
			\&
			\left(\slice{\largecats} \ground\right)\op.
			\arrow[ll, shift left=2, "\mathscr{A}\mathrm{ut}"]
		\end{tikzcd}
	\]
	Of course one could --- indirectly --- obtain the adjunction
	between representations and automorphism groups
	by first taking the monoid of endomorphisms
	and then restricting to groups.
\end{remark}

\section{Small endomorphism monoids}%
\label{section: Refined adjunction in the accessibly tensored case}

When the domain category \( \domain \) of \( F \) is small,
the endomorphism monoid \( \auto F \) is obviously small.
We shall show that this is still the case
when \( \domain \) is large under appropriate accessibility conditions.

\begin{lemma}[Accessible reduction]%
	\label{lemma: accessible reduction}
	Assume that the category \( \ground \) is accessibly enriched%
	~\textnormal{[\cref{def: accessible enrichment}]},
	\( \domain \) is an accessible category
	and \( F \from \domain \to \ground \) is an accessible functor.
	
	Let \( \kappa \) be a small cardinal big enough
	so that \( \domain \) is \( \kappa \)-accessible
	and so that both \( F \) and \( X \mapsto [X,Y] \) commute
	with \(\kappa \)-filtered colimits.
	Let us denote by \( F^\kappa\) the restriction of \(F \)
	to the full subcategory \( \domain^\kappa \subset \domain \)
	of \(\kappa \)-compact objects of \( \domain \).
	Then the canonical map
	\[
		\auto F \longrightarrow \auto {F^\kappa}
	\]
	is an isomorphism. In particular \( \auto F \) is a (small) monoid.
\end{lemma}

\begin{proof}
	Using the universal property of the cointegrals, it is enough to
	show the existence of compatible maps
	\[
		\begin{tikzcd}
			\auto {F^\kappa} \ar[r, "\varphi_X"] & {[F(X),F(X)]}
		\end{tikzcd}
	\]
	for every \( X \in \domain \),
	such that for every \( \kappa \)-compact \( X^\kappa \),
	the map \( \varphi_{X^\kappa} \) is equal to the projection map
	\( \pi_{X^\kappa}\from \auto {F^\kappa}
	\to [F(X^\kappa), F(X^\kappa)] \).
	
	Since every \( X \in \domain\)
	is canonically the \(\kappa \)-filtered colimit
	\( X \isonat \colim_{X^\kappa \to X} X^\kappa \)
	of the \( \kappa \)-compact objects over it,
	\[
		[F(X),F(X)]
		\isonat \lim_{X^\kappa \to X}[F(X^\kappa),F(X)].
	\]
	Every map \( g \from X^\kappa \to X \) induces a morphism
	\[
		\begin{tikzcd}
			\auto {F^\kappa} \ar[r, "\pi_{X^\kappa}"]
			& {[F(X^\kappa),F(X^\kappa)]} \ar[r, "g_\ast"]
			& {[F(X^\kappa),F(X)]}
		\end{tikzcd}
	\]
	and given \( h \from \ul{X}^\kappa \to X^\kappa \), one can
	draw a commutative diagram
	\[
		\begin{tikzcd}
			\auto {F^\kappa}
			\rar["\pi_{X^\kappa}"]
			\dar["\pi_{\ul{X}^\kappa}"']
			&
			{[F(X^\kappa),F(X^\kappa)]}
			\dar["h^\ast"]
			\rar["g_\ast"]
			&
			{[F(X^\kappa),F(X)]}
			\dar["h^\ast"]
			\\
			{[F(\ul{X}^\kappa),F(\ul{X}^\kappa)]}
			\rar["h_\ast",swap]
			&
			{[F(\ul{X}^\kappa),F(X^\kappa)]}
			\rar["g_\ast",swap]
			&
			{[F(\ul{X}^\kappa),F(X)]}
		\end{tikzcd}
	\]
	where the commutation of the first square is guaranteed by the
	universal property of \( \auto {F^\kappa} \).
	This shows that we get a well-defined morphism \( \varphi_X \)
	for every \(X \in \domain \).
	
	By construction of \( \varphi_X \), the following diagram commutes
	\[
		\begin{tikzcd}[ampersand replacement=\&]
			\auto {F^\kappa}
			\arrow[r, "\varphi_X"]
			\arrow[d, "\pi_{X^\kappa}", swap]
			\& {[F(X), F(X)]}
			\arrow[d, "g^\ast"] \\
			{[F(X^\kappa), F(X^\kappa)]}
			\arrow[r, "g_\ast", swap]
			\& {[F(X^\kappa), F(X)]},
		\end{tikzcd}
	\]
	hence when \( g\) is the identity
	of a \( \kappa \)-compact object \( X^\kappa \),
	we get \( \pi_{X^\kappa} = \varphi_{X^\kappa} \) as promised.
	
	Let \( f \from X \to Y \) be a morphism in \( \domain \).
	We need to check the commutativity of the induced square
	\[
		\begin{tikzcd}[ampersand replacement=\&]
			\auto {F^\kappa}
			\arrow[r, "\varphi_X"]
			\arrow[d, "\varphi_Y", swap]
			\& {[F(X),F(X)]}
			\arrow[d, "f_\ast"] \\
			{[F(Y),F(Y)]}
			\arrow[r, "f^\ast", swap]
			\& {[F(X), F(Y)]}.
		\end{tikzcd}
	\]
	By accessibility again,
	one may check the equality \(f^\ast \varphi_Y = f_\ast \varphi_X\)
	after projection
	\( g^\ast \from [F(X), F(Y)] \to [F(X^\kappa),F(Y)] \)
	for every \( g \from X^\kappa \to X \). 
	Then by the commutativity of the diagrams
	\[
		\begin{tikzcd}[ampersand replacement=\&]
			\auto {F^\kappa} \rar["\varphi_X"] \dar["\pi_{X^\kappa}"']
			\& {[F(X),F(X)]} \dar["f_\ast g^\ast"]
			\\
			{[F(X^\kappa),F(X^\kappa)]}
			\rar["f_\ast g_\ast",swap] \& {[F(X^\kappa),F(Y)]}
		\end{tikzcd}
	\]
	and
	\[
		\begin{tikzcd}[ampersand replacement=\&]
			 \auto {F^\kappa} \rar["\varphi_Y"] \dar["\pi_{X^\kappa}"']
			\& {[F(Y), F(Y)]} \dar["{(fg)}^\ast"]
			\\
			{[F(X^\kappa),F(X^\kappa)]} \rar["{(fg)}_\ast",swap]
			\& {[F(X^\kappa),F(Y)]},
		\end{tikzcd}
	\]
	we may conclude the desired result.
\end{proof}

\begin{remark}[Accessibility of the category of representations]
	In view of the previous reduction lemma,
	one may wonder whether \( \forget M \from \rep M \to \ground \)
	is an accessible functor between accessible categories
	whenever \( \ground \) is accessibly enriched.
	
	This appears to be an intricate question in general:
	it is still unknown whether the category of bigebras
	over some well-known props are actually accessible.
	In the particular case where \( \ground \) is accessibly
	tensored (or cotensored), this question receives a positive answer.
	We shall give more details about this case in the next section.
	
	Cogebras over a dg-operad in characteristic zero give
	an example of an accessibly enriched context%
	~[\cref{appendix: cogebras}] that is neither tensored,
	nor cotensored, in which \( P\text{-}\mathsf{cog} \) is accessible
	for any dg-operad \( P \)%
	~\cite{arXiv:180301376L}. 
\end{remark}

\section{The case of tensored enrichment}

In the case where \( \ground \) is tensored over \( \moncat \),
the additional structure allows one to say more
about the adjunction between representations and endomorphisms,
particularly when the tensor structure is well-behaved.

\subsection{The adjunction in the accessibly tensored case}
In the case where forgetful functors are accessible,
we no longer need to have jumps in sizes and we get a refined adjunction
with the category of \emph{small} monoids.

\begin{proposition}[Accessibly tensored case]%
	\label{proposition: accessibly tensored case}
	Assume that \( \ground \) is accessibly tensored over \( \moncat \).
	Then there is an adjunction
	\[
		\begin{tikzcd}[ampersand replacement=\&]
			\monoids{\moncat} \arrow[rr, shift left=2,"\repfunctor"]
			\& \&
			\left(\slice{\accesscats} \ground\right)\op
			\arrow[ll, shift left=2, "\autofunctor"]
		\end{tikzcd}
	\]
	in which \( \accesscats \)
	is the very large category of large accessible categories
	and accessible functors.
\end{proposition}

For this one restricts the adjunction
\( \repfunctor \adj \autofunctor \)
using accessible reduction%
~[\cref{lemma: accessible reduction}] and the following lemmas.

\begin{lemma}%
	\label{thm: accessible category of rep}
	If \( \ground \) is accessibly tensored, then for every
	monoid \( M \), the category of representations \( \rep M \) is
	accessible and the forgetful functor
	\[
		\begin{tikzcd}
			\rep M \ar[rr, "\forget M"] && \ground
		\end{tikzcd}
	\]
	is accessible.
\end{lemma}

\begin{proof}
	Because the functor \( M \mapsto (M \otimes -) \) is monoidal%
	~[See \cref{definition: tensored category}],
	each monoid \( M \) induces an accessible monad
	\( \widetilde M \)
	with underlying functor \( X \mapsto M \otimes X \).
	As a consequence its category of modules is accessible
	and the forgetful functor
	\[
		\begin{tikzcd}
			\catofmod {\widetilde M} \ar[rr, "\forget {\widetilde M}"]
			&& \ground
		\end{tikzcd}
	\]
	is accessible.
	
	We now claim that there is a canonical equivalence of categories
	\[
		\rep M \isonat \catofmod {\widetilde M},
	\]
	compatible with the forgetful functors.
	Let \( (X,\alpha) \) be a representation of \( M \).
	Then the monoid morphism \(\alpha \from M \to [X, X] \)
	is equivalent by adjunction
	to an \( \widetilde M \)-module structure
	\( \widetilde \alpha \from M \otimes X \to X \).
	Let \( (Y, \beta) \) be another representation of \( M \),
	then \( f \from X \to Y \) is a morphism of representations if
	\[
		\begin{tikzcd}[ampersand replacement=\&]
			M
			\arrow[r, "\alpha"]
			\arrow[d, "\beta", swap]
			\& {[X,X]}
			\arrow[d, "f_\ast"] \\
			{[Y,Y]}
			\arrow[r, "f^\ast",swap]
			\& {[X, Y]}
		\end{tikzcd}
	\]
	commutes.
	By adjunction the top right part of the diagram
	is equivalent to \( M\otimes X \to X \to Y \)
	and the bottom left
	is equivalent to \( M \otimes X \to M \otimes Y \to Y \)
	so that the commutativity of the above square
	is equivalent to the commutativity of
	\[
		\begin{tikzcd}[ampersand replacement=\&]
			M \otimes X
			\arrow[rr, "M \otimes f"]
			\arrow[d, "\widetilde \alpha", swap]
			\&\& M \otimes Y
			\arrow[d, "\widetilde \beta"] \\
			X \arrow[rr, "f",swap]
			\&\& Y.
		\end{tikzcd}
	\]
	Hence \( f \from X \to Y \) is a morphism of \( M \)-representations
	if and only if it is a morphism of \( \widetilde M \)-modules.
\end{proof}

\begin{lemma}
	Let \( F \from \domain \to \ground \) be an accessible functor with
	accessible domain.
	Then the counit of the adjunction \(\repfunctor
	\adj \autofunctor\) applied to \( F \)
	\[
		\begin{tikzcd}
			\domain \ar[rr] \ar[dr, "F"'] & &\rep {\auto F}
			\ar[ld, "\forget {\auto F}"]
			\\
			& \ground &
		\end{tikzcd}
	\]
	is given by an accessible functor.
\end{lemma}

\begin{proof}
	The top map of the diagram if accessible because the two other maps
	are accessible%
	~[\cref{thm: accessible category of rep}]
	and the forgetful functor
	\( \rep {\auto F} \to \ground \) is conservative.
\end{proof}

\begin{remark}
	In the accessibly tensored case, the representation functor factors
	through the category of accessible monads on \( \ground \).
	Using an adapted version of a result of Janelidze and Kelly%
	~\cite{zbMATH01687308}, one can show that the adjunction
	\( \repfunctor \adj \autofunctor \) factors as a composite of
	adjunctions
	\[
		\begin{tikzcd}[column sep = huge]
			\monoids{\moncat}
			\arrow[r, shift left=2,"M \longmapsto \widetilde M"]
			&
			\mathsf{Monads}_\mathrm{acc}(\ground)
			\ar[l, shift left=2]
			\ar[r, shift left=2, "\mathsf{Mod}"]
			&
			\left(\slice{\accesscats} \ground\right)\op.
			\arrow[l, shift left=2]
		\end{tikzcd}
	\]
\end{remark}

\subsection{Faithfulness of \texorpdfstring{\( \repfunctor \)}{Rep}}

The question of reconstructing a monoid \( M\) out of its category \(\rep M \)
of representations is an old one, in the Tannakian context for
example%
~[\cref{appendix: Tannaka}].
Such a result cannot be obtained in general without additional
hypotheses.
Instead one can look at the opportunity of recovering \( M \)
as a \emph{submonoid} of \( \auto {\forget M} \).

This is the question of faithfulness
of the \( \repfunctor \) functor which is of independent interest.
As an example, one can view Joyal's results on analytic monads%
~\cite{doi:10.1007/bfb0072514}
as saying in particular that the representation functor is faithful
in the case where \( \ground \) is the
category of sets operadically enriched in symmetric sequences.

The representation functor \( M \mapsto \rep M \)
is a priori not faithful.
A trivial example of this takes \( \ground \) to be the empty category. 
A nontrivial example of independent interest is given by looking
at the functor \( P \mapsto P\text{-}\mathsf{cog} \) mapping a dg-operad
to its category of cogebras. Indeed, one can show that there exists a
non-zero dg-operad without nontrivial cogebras%
~\cite{arXiv:190202551L}:
\[
	\exists~P \neq 0, \quad \cog P = 0.
\]

However, when \( \ground \) is tensored, we get a criterion to check
whether the representation functor is faithful.

\begin{proposition}[Faithfulness of representations]%
	\label{theorem: Rep is faithful}
	Assume that \( \ground\) is faithfully tensored over \(\moncat \),
	then the representations functor
	\[
		\begin{tikzcd}
			\monoids{\moncat} \ar[rr, "\repfunctor"] &&
			\left(\slice{\largecats}{\ground}\right)\op
		\end{tikzcd}
	\]
	is faithful.
	Equivalently, for every monoid \( M \), the unit map
	\[
		M \longrightarrow \auto {\forget M}
	\]
	is a monomorphism.
\end{proposition}

\begin{proof}
	Let \( \phi, \psi \from M \rightrightarrows N \)
	be two morphisms of monoids such that
	\[
		\begin{tikzcd}
			\rep N \ar[rr, "\forget \phi = \forget \psi"] &&\rep M.
		\end{tikzcd}
	\]
	If \( \phi_{!} \) denotes
	the (partially defined) left adjoint to \(\forget \phi\)
	and \( \psi_{!} \) the (partially defined) left adjoint
	to \(\forget \psi\),
	then one has \( \phi_{!} = \psi_{!} \).
	Let \( X\) be an object of \(\ground\),
	because \(\ground \) is tensored
	over \( \moncat\), the monoid \(M\) acts on \(M \otimes X\)
	and \(M \otimes X \) is then the free representation
	of \( M \) induced on \( X \).
	The same goes for \( N \otimes X \).
	As a consequence, one has
	\[
		\forget \phi \circ \phi_{!}(M \otimes X) \isonat
		\forget \psi \circ \psi_{!}(M \otimes X) \isonat N \otimes X.
	\]
	Using the units of the adjunctions, one then gets that
	\[
		\begin{tikzcd}
			M \otimes X
			\ar[rrr, "\phi \otimes X = \psi \otimes X"]
			&&&
			N \otimes X.
		\end{tikzcd}
	\]
	Since this is true for every \( X\), we get \(\phi = \psi \).
\end{proof}

\appendix

\section{Terminology of enriched categories}%
\label{Appendix: enriched categories}

We let the reader turn to Kelly%
~\cite{isbn:0-521-28702-2}
for a detailed exposition on categories enriched
in a monoidal category \( (\moncat, \otimes, \monu) \).
In order to not be bothered by size issues, we fix once and for all
three infinite
inaccessible cardinals \( \textrm{L} < \textrm{XL} < \textrm{XXL} \)
and use the dictionary
\[
	\text{small} \coloneqq \textrm{L}\text{-small}; \quad
	\text{large} \coloneqq \textrm{XL}\text{-small}; \quad
	\text{very large} \coloneqq \textrm{XXL}\text{-small}.
\]
We now assume that \( \moncat \) is large (has large sets of objects
and morphisms) and has all small limits and colimits. In what follows
we consider a large \( \moncat \)-enriched category
\[
	\begin{tikzcd}
		\ground\op \times \ground \ar[rr,"{[-,-]}"] && \moncat
	\end{tikzcd}
\]
and assume that \( \ground \) is large.

\subsection{Enlargement of the universe}%
\label{Appendix: enlargement of the universe}

For convenience (when computing over large diagrams), we shall enlarge
\( \moncat \): we choose a very large monoidal category
\( (\biguni, \otimes, \monu) \) with a full monoidal embedding
\[
	\begin{tikzcd}
		(\moncat, \otimes, \monu) \ar[rr, hook] &&
		\left(\biguni, \otimes, \monu\right).
	\end{tikzcd}
\]
The enlarged universe can be chosen to
be locally large, have all large limits and colimits and the embedding
can be assumed to commute with small limits and colimits.
This is discussed for example by Kelly%
~\cite[\S2.6]{isbn:0-521-28702-2} (albeit in the closed symmetric
setting).

The \( \moncat \)-category \( \ground \) can now without effort
be seen as a \( \biguni \)-category
\[
	\begin{tikzcd}
		\ground\op \times \ground
		\ar[rr,"{[-,-]}"] && \moncat \ar[r, hook]
		&\biguni.
	\end{tikzcd}
\]

\subsection{Properties of enrichments}

\begin{definition}[Closed monoidal category]%
	\label{definition: closed}
	One says that \( \moncat \) is \emph{closed} when the functor
	\( Y \mapsto Y \otimes X \) has a right adjoint \( Z \mapsto X^Z \)
	for each object \( X \) in \( \moncat \).
\end{definition}

\begin{definition}[Tensored]%
	\label{definition: tensored category}
	One says that \( \ground \) is \emph{tensored} over \( \moncat \)
	whenever \( \moncat \) is closed
	and for every \( X \in \ground \) and \( M \in \moncat \),
	the functor
	\[
		Y \longmapsto {[X,Y]}^M
	\]
	is \( \moncat \)-representable by an object
	denoted \( M \otimes X \in \ground \).
	In that case, since \( \moncat \) is closed the induced functor
	\[
		\begin{tikzcd}
			(\moncat, \otimes, \monu)
			\ar[rrr, "M \longmapsto (M \otimes -)"]
			&&& (\fun(\ground,\ground), \circ, \id \ground)
		\end{tikzcd}
	\]
	is naturally endowed with a monoidal structure.
\end{definition}

\begin{definition}[Faithfully tensored]
	We shall say that \( \ground \) is \emph{faithfully} tensored over
	\( \moncat \) if it is tensored and the functor
	\[
		\begin{tikzcd}
			\moncat \ar[rrr, "M \longmapsto (M \otimes -)"]
			&&& \fun(\ground,\ground)
		\end{tikzcd}
	\]
	is faithful.
\end{definition}

\begin{definition}[Accessibly tensored]
	We shall say that \( \ground\) is \emph{accessibly} tensored
	over \(\moncat \)
	if it is tensored,
	both \( \moncat\) and \(\ground \) are accessible
	and for every \( M\in \moncat \), the functor
	\[
		\begin{tikzcd}
			\ground \ar[rrr, "X \longmapsto M \otimes X"] &&& \ground
		\end{tikzcd}
	\]
	is accessible.
\end{definition}

\begin{definition}[Accessibly enriched]%
	\label{def: accessible enrichment}
	When \( \moncat \) and \( \ground \) are both accessible,
	we shall say that \( \ground \) is \emph{accessibly enriched}
	if the exists a small cardinal \( \kappa \)
	such that for every \(Y \in \ground \), the functor
	\[
		\begin{tikzcd}
			\ground\op \ar[rrr, "X \longmapsto {[X, Y]}"] &&& \moncat
		\end{tikzcd}
	\]
	commutes with \( \kappa \)-cofiltered limits.
\end{definition}

\begin{remark}
	One can check that if \( \ground \) is accessibly tensored,
	it is then accessibly enriched.
\end{remark}

\section{Examples of contexts of application}%
\label{Appendix: contexts}

In this appendix, we give several application contexts for the
adjunction
\[
	\begin{tikzcd}[ampersand replacement=\&]
		\monoids{\biguni} \arrow[rr, shift left=2,"\repfunctor"]
		\&
		\&
		\left(\slice{\largecats} \ground\right)\op.
		\arrow[ll, shift left=2, "\autofunctor"]
	\end{tikzcd}
\]
In each context, the terminology is specific, both for monoids and
for their categories of representations.

\subsection{Using a closed symmetric monoidal category}

In the next examples, we fix a presentable closed symmetric monoidal
category \( (\ground, \otimes, \monu) \)
and denote its internal hom by \( \langle - ,- \rangle \).
We then consider several enrichments for \( \ground \).

Potential examples of such closed symmetric monoidal categories include
the category of sets, vector spaces or
coassociative cogebras (more generally cogebras over Hopf operads).
It also includes the categories of sheaves valued in those categories.

\subsubsection{Self enrichment}

This one is the most obvious, since the monoidal structure of
\( \ground \) is closed, it is self-enriched via
\[
	[X, Y] \coloneqq \langle X, Y \rangle.
\]
In this context, the general idea of the adjunction
\( \repfunctor \adj \autofunctor \) was 
well-known to people doing reconstruction theorems à la Tannaka.
It appears for example in Street's
\emph{Quantum groups: a path to current algebra}%
~\cite[Ch.~16]{isbn:978-0521695244}.

\subsubsection{Operadic enrichment}
Let us denote by \( \ground^{\symgroup{}\op} \)
the category of symmetric sequences:
sequences of objects \( M(n) \) of \( \ground \) endowed with
right \( \symgroup n \)-actions for every natural \( n \).
The category \( \ground \) is accessibly tensored
over the category of symmetric sequences via the formula
\[
	M \triangletimes X \coloneqq \coprod_{n \in \naturals} M(n)
	\otimes_{\symgroup n} X^{\otimes n}.
\]
This induces a monoidal structure on symmetric sequences
\[
	M \triangletimes N \coloneqq \coprod_{n \in \naturals} M(n)
	\otimes_{\symgroup n} N^{\convolution n}.
\]
Where \( \convolution \) denotes the convolution of symmetric sequences.
The associated enrichment is given by
\[
	[X, Y](n) \coloneqq \langle X^{\otimes n}, Y \rangle.
\]
Monoids in symmetric sequences are called operads
\[
	\mathsf{Op}(\ground) \coloneqq \monoids {\ground^{\symgroup{}\op},
	\triangletimes, \monu_{\triangletimes}}
\]
Given an operad \( P \), its category of representations is called the
category of \( P \)-algebras.
One thus gets an adjunction
\[
	\begin{tikzcd}[ampersand replacement=\&]
		\mathsf{Op}(\ground) \arrow[rr, shift left=2,"\mathsf{Alg}"]
		\& \&
		\left(\slice{\accesscats} \ground\right)\op.
		\arrow[ll, shift left=2, "\autofunctor"]
	\end{tikzcd}
\]

\subsubsection{The other (cogebraic) operadic enrichment}%
\label{appendix: cogebras}

This time we let \( \ground^{\symgroup{}} \) be the category of
symmetric sequences with left actions of the symmetric groups.
It admits a monoidal structure given by
\[
	M \optriangletimes N \coloneqq \coprod_{n \in \naturals}
	M^{\convolution n} \otimes_{\symgroup n} N(n)
\]
and the associated enrichment is
\[
	[X, Y](n) \coloneqq \langle X, Y^{\otimes n} \rangle.
\]
Since left and right actions of symmetric groups are equivalent, one
has an equivalence of categories
\[
	\monoids {\ground^{\symgroup{}},
	\optriangletimes, \monu_{\optriangletimes}} \isonat
	\monoids {\ground^{\symgroup{}\op},
	\triangletimes, \monu_{\triangletimes}} \isonat
	\mathsf{Op}(\ground).
\]
In this case,
the category of representations of an operad \( P \)
is its category of cogebras.
Conversely,
the functor \( \autofunctor \) associates to a functor \( F \),
seen as an object of the functor category,
its \emph{coendomorphism operad}.

In general, the category of \( P \)-cogebras may not 
be presentable, although (for example) it is presentable if the ground
category is dg-vector spaces%
~\cite{arXiv:180301376L}.
Thus, one has the adjunction
\[
	\begin{tikzcd}[ampersand replacement=\&]
		\mathsf{Op}(\mathsf{dgVect})
		\arrow[rr, shift left=2,"\mathsf{Cog}"]
		\& \&
		\left(\slice{\accesscats} {\mathsf{dgVect}}\right)\op.
		\arrow[ll, shift left=2, "\autofunctor"]
	\end{tikzcd}
\]

This example arises naturally in applications and appeared, for example,
in unpublished work by May, who considered it well-known. 

In one application, the singular chains functor from 
topological spaces to chain complexes factors through the 
category of \( \einfinity \)-cogebras in chain complexes.

The following stable improvement of this example was pointed out to
us by Arone: the coendomorphism operad of the suspension functor from
pointed spaces to spectra can be shown to be weakly equivalent to the
commutative operad%
~\cite{doi:10.24033/ast.904}.

\subsubsection{Propic enrichments}

Going further, one can enrich \( \ground \) in the category
of bisymmetric sequences
\( \ground^{\symgroup{}\op \times \symgroup{}} \) using
\[
	[X, Y](p, q) \coloneqq \langle X^{\otimes p}, Y^{\otimes q}\rangle.
\]
There are several monoidal structures on bisymmetric sequences
compatible with these enrichment objects, depending on the classes of
graphs involved in the definition of the monoidal structure. One can
allow connected graphs, in which 
case the monoids are properads%
~\cite[2.1]{doi:10.1016/j.crma.2004.04.004}, 
or allow only simply connected graphs, in which case the monoids are 
dioperads%
~\cite[4.2]{doi:10.4310/mrl.2003.v10.n1.a11}. 
Similar but more exotic examples are also possible%
~\cite{doi:10.1090/surv/203}.

\subsection{Examples with exogenic enrichments}

\subsubsection{Representations of topological monoids}%
\label{appendix: Tannaka}
The following example is taken from the duality between 
topological groups and their categories of representations 
due to Tannaka%
~\cite{doi:10.1007/bfb0084235}.
The category of finite dimensional vector spaces is canonically
enriched in topological spaces. Since this category is small, one
gets an adjunction
\[
	\begin{tikzcd}[ampersand replacement=\&]
		\monoids{\mathsf{Top}}
		\arrow[rr, shift left=2,"\mathsf{Rep}_\mathrm{fd}"]
		\& \&
		\left(\slice{\mathsf{Cat}} {\mathsf{Vect}_\mathrm{fd}}\right)\op
		\arrow[ll, shift left=2, "\autofunctor"]
	\end{tikzcd}
\]
where \( \autofunctor \) associates
to any functor \( F \from \domain \to \mathsf{Vect}_\mathrm{fd} \)
its topological monoid of endomorphisms.

\subsubsection{Bigebras}
Let \( \fieldk \) be a field.
The category of associative \( \fieldk \)-algebras
is naturally cotensored over \( \fieldk \)-cogebras: given a cogebra
\( V \) and an algebra \( \Lambda \), convolution gives
\( \hom {\fieldk} V \Lambda \) a structure of associative algebra.
This cotensorization comes with an enrichment and a tensorization%
~\cite{arXiv:1309.6952}.

Monoid objects in cogebras are bigebras.
Given a bigebra \( H \),
it is an exercise to verify that the category of representations
\( \rep H \) is naturally isomorphic to
the category of \( H \)-module algebras
studied by Hopf theorists%
~\cite[4.1.1]{doi:10.1090/cbms/082}
equipped with the functor to algebras forgetting the \( H \)-module
structure.
We thus obtain an adjunction
\[
	\begin{tikzcd}[ampersand replacement=\&]
		\mathsf{Bigebras} \arrow[rr, shift left=2,"\mathsf{Mod}"]
		\& \&
		\left(\slice{\accesscats} {\mathsf{Alg}}\right)\op
		\arrow[ll, shift left=2, "\autofunctor"]
	\end{tikzcd}
\]
where for an accessible functor \( F \from \domain\to \mathsf{Alg} \),
the endomorphism bigebra \( \autofunctor(F) \)
is universal among bigebras acting compatibly
on the objects of \( \domain \).

\section*{Acknowledgements}

The authors would like to thank Rune Haugseng, Theo Johnson-Freyd,
Johan Leray, Emily
Riehl, and Claudia Scheimbauer for useful discussions, as well as 
Greg Arone and Birgit Richter for pointing us to relevant literature.

\bibliography{ms.bbl}
\bibliographystyle{style}

\end{document}

%% file: config.tex
\title{Representations are adjoint to endomorphisms}

\author{Gabriel C. Drummond-Cole}
\thanks{The first and third authors were supported by \texttt{IBS-R003-D1.}}
\address{Center for Geometry and Physics, Institute for Basic Science
	(IBS), Pohang, Republic of Korea 37673}
\email{gabriel@ibs.re.kr}

\author{Joseph Hirsh}
\thanks{The second author was supported by \texttt{NSF DMS-1304169}}
\email{josephhirsh@gmail.com}

\author{Damien Lejay}
\address{Center for Geometry and Physics, Institute for Basic Science
	(IBS), Pohang, Republic of Korea 37673}
\email{lejay@paracompact.space}

\date{}

\hypersetup{%
	pdftitle={Representations are adjoint to endomorphisms},
	pdfauthor={Gabriel C. Drummond-Cole, Joseph Hirsh and Damien Lejay},
	pdfsubject={},
	pdfkeywords={},
	pdfcreator={pdfLaTeX},
	pdfproducer={LaTeX with hyperref}
}

\usetikzlibrary{fit}
\newcolumntype{?}[2]{!{\textcolor{#2}{\vrule width #1}}}

%% file: macros.tex
\newcommand{\from}{\colon}
\newcommand{\adj}{\mathrel{\dashv}}

\newcommand{\ground}{\mathscr{C}}
\newcommand{\domain}{\mathcal{D}}

\newcommand{\fun}{\operatorname{Fun}}

\newcommand{\slice}[2]{#1_{/#2}}

\DeclareFontFamily{U}{min}{}
\DeclareFontShape{U}{min}{m}{n}{<-> udmj30}{}
\newcommand{\sectionref}[1]{\hyperref[#1]{\S\thinspace\ref{#1}}}
\newcommand{\morita}[1]{\mathsf{Mor}_*}

\DeclareMathOperator*{\colim}{colim}

\newcommand{\fieldk}{\boldsymbol{K}}

\newcommand{\coint}[1]{\int_{#1}^{\ast}}
\newcommand{\forget}[1]{\mathrm{U}_{#1}}
\newcommand{\convolution}{\circledast}
\newcommand{\symgroup}[1]{\mathbf{S}_{#1}}
\newcommand{\triangletimes}{\mathbin{\triangleleft}}
\newcommand{\optriangletimes}{\mathbin{\triangleright}}
\newcommand{\isonat}{=}

\newcommand{\pointedbim}[1]{\mathsf{Bimod}_\bullet\mleft(#1\mright)}

\DeclareMathOperator{\Aut}{Aut}
\newcommand{\auto}[1]{\mathscr{E}(#1)}

\newcommand{\autofunctor}{\mathscr{E}}

\newcommand{\rep}[1]{#1\text{-}\mathsf{rep}}
\newcommand{\cog}[1]{#1\text{-}\mathsf{cog}}

\newcommand{\repfunctor}{\mathsf{Rep}}

\newcommand{\largecats}{\widehat{\mathsf{Cat}}}
\newcommand{\accesscats}{\mathsf{Acc}}
\newcommand{\moncat}{\mathscr{V}}
\newcommand{\biguni}{\mathscr{\widehat{V}}}
\newcommand{\monu}{\mathbf{1}}

\newcommand{\biggroups}{\widehat{\mathsf{Grp}}}
\newcommand{\monoids}[1]{\mathsf{Mon}\mleft(#1\mright)}

\newcommand{\bigsets}{\widehat{\mathsf{Sets}}}
\newcommand{\abelian}{\mathsf{Ab}}

\renewcommand{\op}{^\mathrm{op}}
\renewcommand{\hom}[3]{\mathrm{Hom}_{#1}(#2, #3)}
\newcommand{\nat}[3]{\mathrm{Nat}_{#1}(#2, #3)}
\newcommand{\catofmod}[1]{#1\text{-}\mathsf{mod}}

\newcommand{\einfinity}{\mathsf{E}_\infty}

\DeclareMathOperator{\Hom}{Hom}

\newcommand{\ul}{\underline}

%% file: ms.bbl
\providecommand{\href}[2]{#2}\begingroup\raggedright\begin{thebibliography}{10}\footnotesize

\bibitem{arXiv:190609006D}
Gabriel~C. Drummond-Cole, Joseph Hirsh, and Damien Lejay, `{Endomorphism
  operads of functors}', {\em ArXiv e-prints} (June, 2019) {},
  \href{http://arxiv.org/abs/1906.09006}{{\ttfamily arXiv:1906.09006
  [math.CT]}}.

\bibitem{On_Ext_and_exact_sequences}
Nobuo Yoneda, `{On Ext and exact sequences}', {\em Journal of the Faculty of
  Science, Imperial University of Tokyo} {\bfseries 8} (1960) 507--576.

\bibitem{isbn:978-1-4704-2723-8}
Donald Yau, {\em Colored operads}, vol.~170 of {\em Graduate Studies in
  Mathematics}.
\newblock American Mathematical Society, Providence, RI, 2016.

\bibitem{arXiv:180301376L}
Brice~Le Grignou and Damien Lejay, `{Homotopy theory of linear cogebras}', {\em
  ArXiv e-prints} (Mar., 2018) {},
  \href{http://arxiv.org/abs/1803.01376}{{\ttfamily arXiv:1803.01376
  [math.AT]}}.

\bibitem{zbMATH01687308}
G.~{Janelidze} and G.~M. {Kelly}, `A note on actions of a monoidal category',
  {\em Theory and Applications of Categories} {\bfseries 9} (2001) 61--91.

\bibitem{doi:10.1007/bfb0072514}
André Joyal, \href{http://dx.doi.org/10.1007/bfb0072514}{`Foncteurs
  analytiques et espèces de structures',} in {\em Combinatoire énumérative},
  Gilbert Labelle and Pierre Leroux, eds., pp.~126--159.
\newblock Springer Berlin Heidelberg, Berlin, Heidelberg, 1986.

\bibitem{arXiv:190202551L}
Brice~Le Grignou and Damien Lejay, `{Operads without cogebras}', {\em ArXiv
  e-prints} (Feb., 2019) {}, \href{http://arxiv.org/abs/1902.02551}{{\ttfamily
  arXiv:1902.02551 [math.AT]}}.

\bibitem{isbn:0-521-28702-2}
G.~M. Kelly, {\em Basic Concepts of Enriched Category Theory}, vol.~64 of {\em
  London Mathematical Society Lecture Note Series}.
\newblock Cambridge University Press, Cambridge, New York, 1982.

\bibitem{isbn:978-0521695244}
Ross Street, {\em Quantum Groups: a path to current algebra}, vol.~19 of {\em
  Australian Mathematical Society Lecture Series}.
\newblock Cambridge University Press, 2007.

\bibitem{doi:10.24033/ast.904}
Greg Arone and Michael Ching, \href{http://dx.doi.org/10.24033/ast.904}{{\em
  Operads and chain rules for the calculus of functors}}, vol.~338 of {\em
  Ast{\'e}risque}.
\newblock Soci{\'e}t{\'e} Math{\'e}matique de France, 2011.

\bibitem{doi:10.1016/j.crma.2004.04.004}
Bruno Vallette, `Koszul duality for {PROPs}',
  \href{http://dx.doi.org/10.1016/j.crma.2004.04.004}{{\em Comptes Rendus
  Mathématique} {\bfseries 338} no.~12, (June, 2004) 909--914}.

\bibitem{doi:10.4310/mrl.2003.v10.n1.a11}
Wee~Liang Gan, `Koszul duality for dioperads',
  \href{http://dx.doi.org/10.4310/mrl.2003.v10.n1.a11}{{\em Mathematical
  Research Letters} {\bfseries 10} no.~1, (2003) 109--124}.

\bibitem{doi:10.1090/surv/203}
Mark Johnson and Donald Yau, \href{http://dx.doi.org/10.1090/surv/203}{{\em A
  Foundation for {PROPs}, Algebras, and Modules}}.
\newblock American Mathematical Society, May, 2015.

\bibitem{doi:10.1007/bfb0084235}
Andr{\'{e}} Joyal and Ross Street,
  \href{http://dx.doi.org/10.1007/bfb0084235}{`An introduction to Tannaka
  duality and quantum groups',} in {\em Lecture Notes in Mathematics},
  pp.~413--492.
\newblock Springer Berlin Heidelberg, 1991.

\bibitem{arXiv:1309.6952}
Matthieu Anel and André Joyal, `{Sweedler Theory for (co)algebras and the
  bar-cobar constructions}', {\em ArXiv e-prints} (Sept., 2013) {},
  \href{http://arxiv.org/abs/1309.6952}{{\ttfamily arXiv:1309.6952 [math.CT]}}.

\bibitem{doi:10.1090/cbms/082}
Susan Montgomery, \href{http://dx.doi.org/10.1090/cbms/082}{{\em Hopf Algebras
  and Their Actions on Rings}}.
\newblock American Mathematical Society, Oct., 1993.

\end{thebibliography}\endgroup
